\documentclass[12pt]{amsart}
\usepackage{epsfig,color}
\usepackage{blindtext}
\usepackage{hyperref}
\usepackage{graphicx}
\usepackage{enumitem} 
\usepackage{url}
\usepackage{amssymb}
\usepackage{graphicx,import}
\usepackage{comment}
\usepackage{ulem}
\usepackage{outlines}
\usepackage{esint}
\usepackage{verbatim}
\usepackage{amsrefs}

\theoremstyle{definition}

\newtheorem{theorem}{Theorem}[section]
\newtheorem{definition}[theorem]{Definition}

\newtheorem{lemma}[theorem]{Lemma}
\newtheorem{remark}[theorem]{Remark}
\newtheorem{corollary}[theorem]{Corollary}
\newtheorem{example}[theorem]{Example}

\newtheorem{property}[theorem]{Property}

\newtheorem*{remark*}{Remark}

\numberwithin{equation}{section}

\newcommand{\mc}{\mathcal}
\newcommand{\mb}{\mathbb}

\newcommand{\pa}{\partial}
\newcommand{\mr}{\mathrm}
\newcommand{\lag}{\langle}
\newcommand{\rag}{\rangle}

\DeclareMathOperator{\dist}{dist}

\title{Boundary Harnack Principle on Nodal Domains}
\date{\today}

\author{Fanghua Lin}
\address{Courant Institute of Mathematical Sciences, New York University, 251 Mercer Street, New York, NY10012, USA}
\email{linf@cims.nyu.edu}

\author{Zhengjiang Lin}
\address{Courant Institute of Mathematical Sciences, New York University, 251 Mercer Street, New York, NY10012, USA}
\email{malin@nyu.edu}

\begin{document}

\begin{abstract}
We study some geometric and potential theoretic properties of nodal domains of solutions to certain 
uniformly elliptic equations. In particular, we establish corkscrew conditions, Carleson type estimates and boundary 
Harnack inequalities on a class of nodal domains.
\end{abstract}

\maketitle

\textbf{Dedicate to Professor Jiaxing Hong on the Ocassion of His 80th Birthday}
\section{Introduction}

Let $w$ be a nonzero solution of $\mc{L}(w) = 0$ in $B_{10}(0) \subset \mb{R}^n$, where $\mc{L} = \pa_i (a_{ij}(x) \pa_j)$ with the coefficient matrix $A(x) = (a_{ij})$ satisfying (\ref{Coefficient 1}) and (\ref{Coefficient 2}).
Let $\Omega$ be a nodal domain of $w$, which is a path-connected subregion of the set $\{x \in B_{10} \ | \ w(x) \neq 0\}$.
In order to get meaningful analytic estimates such as those presented in survey articles 
\cite{A12},\cite{T17}, one cannot avoid dealing with the cases that $\Omega$ is 
a non-smooth domain. At a micro scale,  $\Omega$ resembles cone like structures near each point of $\pa 
\Omega \cap B_{10}$  by the unique continuation property.
While at larger scales, $\Omega$ could be like a highly twisted H\"{o}lder type domain with rather 
complicated geometrical and topological properties. In higher dimensions, even when the nodal set $Z(w) = 
\{ x \in B_{10} \ | \ w(x) = 0\}$ is in a small neighborhood of a one dimensional smooth set and hence small in apparent
geometric size, its complexity is hard to bound. For example, by Runge's theorem, one can easily 
construct a sequence of harmonic functions $\{w_k(x)\}$ in $\mb{R}^n$ ( $n \geq 2$ ) such that $w_k \to -1$ locally uniformly on $\Sigma$ while $w_k \to +1$ locally uniformly on $\mb{R}^n \backslash \Sigma$, 
where $\Sigma$ is a finite union of closed half-lines connecting the origin to 
infinity. In particular, some of
the nodal domains of $w_k$ inside $B_{10}$ are collapsed into an arbitrarily small open neighborhood of 
$\Sigma$. In such cases, one cannot expect the validity of a three sphere theorem for solutions or the validity of an uniform Carleson type estimate or the boundary Harnack principle. It is remarkable, on the other hand, that Logunov \cite{L182} proved the
Nadirashvili's conjecture, which asserts that $H^{n-1} (\{x \ | \ w(x) = 0\} \cap B_1) \geq C(n) >0$ for a harmonic function $w$ with $w(0)=0$. It means that such sequences of harmonic functions
$\{w_k(x)\}$ as described above must be highly oscillating and not locally uniformly bounded.

After examining various examples one concludes that in order to carry out classical potential and elliptic PDE analysis on a nodal domain $\Omega$ similar to those in well-known cases of Lipschitz and NTA domains, see  
\cite{CFMS81}, \cite{JK82}), one needs to make some additional assumptions on the solutions $w$ and
operators $\mc{L}$. In particular, one hopes to find a class of domains that are
invariant under scaling (at least, they are so with respect to the scaling ups). In recent 
works of Logunov and Malinnikova \cite{LM15}, \cite{LM16}, it is proved that if $u,v$ 
are usual harmonic functions in $B_{10}$ with $Z(u) = Z(v)$, then the ratio $f \equiv v/u$ is analytic and satisfies the 
Harnack inequality and $|\nabla f|$ as well as higher order derivatives of $f$ validate  
estimates like those for typical solutions of elliptic PDEs with analytic coefficients. Similar results were proved in $\mb{R}^2$ in \cite{M14}. 
All these estimates depend on a fixed nature of the analytic variety $Z(u)$, and they are not necessarily scaling invariant. On the other hand, it is not hard to see that 
\cite{LM16} can be generalized to the case when $u,v$ are solutions of elliptic PDEs with 
real analytic coefficients.

In this paper, we consider a class of solutions $w$ which have a fixed bound on their growth rate 
or a bound on their frequencies on $B_{10}$ (see Section \ref{Section PRE} for details). More precisely, we shall consider those 
$w \in \mc{S}_{N_0} (\Lambda)$ defined by (\ref{Definition of Compact 
Family}), a very natural class of solutions which have been investigated in great detail for their 
quantitative unique continuation properties and related geometric measure estimates on the nodal and 
critical point sets (see Section \ref{Section PRE}). The following are main results of this paper.

\subsection{Main Results}

\begin{theorem}\label{Main Holder}
    \textit{
    Suppose that $\mc{L}u=\mc{L}v = 0$ in $B_{10}$, $N_u \leq N_0 < \infty$, and $0 \in Z(u) \subset Z(v)$, then $v/u \in C^\alpha(B_{1})$ for some $\alpha = \alpha(\Lambda,N_0) \in (0,1)$.
    }
\end{theorem}

For constants in the form of $C = C(\Lambda, N_0)$, we mean that the constants depend on 
$N_0$ and the conditions on the coefficients in (\ref{Coefficient 1}) and (\ref{Coefficient 
2}) of the operator $\mc{L}$. Here $N_u$ 
is the frequency function (doubling index) of $u$ on $B_{10}$, which will be reviewed in 
Section \ref{Section PRE}. Various equivalent notations and auxiliary lemmas are discussed
in Section \ref{Section PRE}.

The above theorem is derived, as in earlier works, from the upper bound inequality
\begin{equation}\label{Upper Bound Inequality}
	\sup_{B_1} |v/u| \leq C(\Lambda, N_0) \cdot (\sup_{B_8} |v| / \sup_{B_8} |u|),
\end{equation}
when $Z(u) \subset Z(v)$ and $N_u \leq N_0$. See Theorem \ref{Local Boundedness}. In order to get the H\"{o}lder continuity for $v/u$, one also needs an iterative argument involving improvements of upper and lower bounds as in \cite{CFMS81} and \cite{JK82}). The latter
is based on the following Harnack type estimate: 
\begin{equation}
	(\sup_{B_{1/8}}(v/u) - \inf_{B_{1}}(v/u )) \leq C(\Lambda,N_0) \cdot (\inf_{B_{1/8}}(v/u) - \inf_{B_{1}}(v/u)).
\end{equation}

To prove this Harnack type estimate, we need to show that the frequency of the function $v - u \cdot\inf_{B_1}(v/u)$ is also bounded in a smaller ball like $B_{1/4}$. 

The above leads us to the next more general result which says that, if two solutions of two possibly different elliptic partial differential equations have the same nodal set in $B_{10}$, and if one of the solution has a bounded frequency or 
a fixed growth rate, then the other has to have a bounded frequency and growth rate as well. We remark that, it is in this latter
statement that we require both operators $\mc{L}, \mc{L}_1$ to have Lipschitz continuous coefficients. In fact, it can be shown that the conclusion is not valid if operators are uniformly elliptic with only bounded measurable coefficients.

\begin{theorem}\label{Main Frequency}
    \textit{
    Suppose that $\mc{L}(u) = \mc{L}_1(v) = 0$ in $B_{10}$, and $0 \in Z(u) = Z(v) $. Also assume that $N_u \leq N_0 < \infty$. Then, there is a positive constant $D = D(\Lambda, N_0) < \infty$, such that $N_v(0,1) \leq D$.
    }
\end{theorem}

Here, $N_v(0,1)$ is the frequency function for $v$ and ball $B_1(0)$. We emphasize again that $\mc{L}$ and $\mc{L}_1$ could be two different elliptic operators satisfying (\ref{Coefficient 1}) and (\ref{Coefficient 2}). This provides a local compactness property for a large class of solutions to such elliptic equations,
see \cite{GL86}.  

As a direct corollary of Theorem \ref{Main Holder} and Theorem \ref{Main Frequency}, we have the following theorem.

\begin{theorem}\label{Main Polynomial}
	\textit{
	Suppose that $\Delta(u) = \Delta(v) =0$ in $\mb{R}^n$, $u$ is a harmonic polynomial, and $Z(u) = Z(v)$. Then, there is a constant $c \in \mb{R}\backslash\{0\}$ such that $v = c \cdot u$.
    	}
\end{theorem}

When $u$ is a homogeneous harmonic polynomial, this theorem was proved, see Theorem 1.2 in \cite{LM16}. In Corollary \ref{Liouville Theorem}, we proved, in fact, a bit stronger statement. The condition that $u$ is a polynomial is important for \cite{LM16}. For example, let $u_{a,b}(x,y,z) = \sin(z)e^{ax + by}$ 
 and $a^2 + b^2=1$, then harmonic functions $u_{a,b}$ share the same nodal set, but with exponential growth. The work \cite{LM15} described many interesting examples of harmonic functions sharing the same nodal set either locally or globally.

In connection to harmonic/PDE analysis on non-smooth domains (e.g., \cite{CFMS81} and \cite{JK82}), we also established Carleson type estimates like (\ref{Upper Bound Inequality}) on a single nodal domain $\Omega$ (defined by a solution $u_0$). It should be noted that, in general, one cannot expect continuity (or even boundedness) up to  the boundary $\partial \Omega$ for the ratio $v/u$ if $v$ and $u$ are solutions defined only on this single $\Omega$.
See section \ref{Section Single}, Theorem \ref{Boundary Harnack in Single Domain} and the Example \ref{Example Leon Simon}.
 Another main result we established is the following statement:

\begin{theorem}\label{Main Measure}
    \textit{
    Let $\Omega$ be a nodal domain of a solution $u_0 \in \mc{S}_{N_0}(\Lambda)$ with $\mc{L}_0(u_0) = 0$ 
    and $0 \in \pa \Omega$ . Then, there is a set consisting of bounded number of points $\{x_1 , \dots , x_{T_0}\}$ with 
$T_0 = T_0 (\Lambda,N_0)$ in $\Omega \cap B_2$, such that
        \begin{equation}\label{Absolute Continuity}
            C^{-1} \cdot |\nabla u_0 | \cdot H^{n-1} \llcorner (\pa \Omega \cap B_1) \leq \sum_i ^{T_0} \ \omega_i \llcorner (\pa \Omega \cap B_1) \leq C \cdot | \nabla u_0 |  \cdot H^{n-1} \llcorner (\pa \Omega \cap B_1) 
        \end{equation}
    for some positive constant $C = C(\Lambda, N_0)$, where $\{ \omega_i (\cdot)\}$'s are $\mc{L}_0$-harmonic 
measures on $\pa (\Omega \cap B_5)$ with poles at $x_i \in \Omega \cap B_2$ for $i = 1,2, \dots, T_0$. In particular, $\sum_i ^{T_0} \ \omega_i \llcorner (\pa \Omega \cap B_1)$, $H^{n-1} \llcorner (\pa \Omega \cap B_1) $ and $|\nabla u_0| \cdot H^{n-1} \llcorner (\pa \Omega \cap B_1) $
are mutually absolutely continuous.
    }
\end{theorem}

Note that it is necessary in general to choose more than one of such points $x_i$ and the corresponding 
harmonic measures in order to have the two-sided estimates as shown in the above theorem. If one selects only one of such points (and its associated harmonic measure), then only the right half inequality of (\ref{Absolute Continuity}) is true in general. We shall also point out that, the 
locations of these points $\{x_i\}$, while flexible, may depend on a particular nodal domain (and hence the defining function $u_0$). What is important is that one can always choose such points; moreover the number of such points $\{x_i\}$ is uniformly bounded (by $T_0$) for any nodal domain of $u_0$ for all $u_0 \in \mc{S}_{N_0}(\Lambda)$.
   
There are two basic ingredients in proving these results. One is the validity of the corkscrew condition and the existence of modified 
Harnack chains for this class of nodal domains. Such geometric structural properties make these nodal domains very similar to the NTA domains,
see \cite{JK82}. The other is the Carleson type estimates for solutions as in \cite{CFMS81}, \cite{JK82} and \cite{LM16}. We shall show, by the doubling properties of the defining functions (solutions), that nodal domains possess these desired geometric properties. One then established the Boundary Harnack princple and Carleson type estimates for non-negative solutions of uniformly elliptic operators with bounded measurable coefficients on such domains. The latter may be a useful fact for applications to some elliptic free boundary problems.

To end the descriptions of main results, let us show an example due to Leon Simon, see \cite{HHHN99}.
\begin{example}\label{Example Leon Simon}
	Let $f(z)$ be a smooth function on $\mb{R}$ with $|f''| < 1/2$. The function $ u(x,y,z) = xy+ f(z) $ satisfies the elliptic equation $\partial^2 _{xx} u + \partial^2 _{yy} u + \partial^2 _{zz}  u - (f''(z)) \partial^2_{xy}u = 0$. Then, the singular set of $Z(u)$ is $\{(0,0,z) \in \mb{R}^3 \ \vert \ f(z) = f'(z) = 0 \}$.
\end{example}

One can choose a smooth (even analytic) and sufficiently small $f$ such that around the singular set of $Z(u)$, the $Z(u)$ behaves like many double cones and $u$ only has two nodal domains. The topology of the nodal domains of $u$ and its critical set can be unbounded (in smooth case). While the frequency of the solution $u$ is close to $2$. And Carleson type estimates as well as the boundary Harnack principle are still valid among other conclusions proved here.

\subsection{Structure of the Paper}

In section \ref{Section PRE}, we go over some tools and basic facts that will be used in the paper. In particular, the notions of the frequency function, the doubling index, and the three spheres theorems.
In section \ref{Section Corkscrew}, we will show the corkscrew property and a modified Harnack chain 
property for this class of nodal domains. Our arguments generalize that in \cite{LM16}.
In section \ref{Section BDD}, we will first show that the ratio $v/u$ is locally bounded near the nodal sets, and then give the proofs for Theorem \ref{Main Holder}. We will also discuss the entire solutions and prove Theorem \ref{Main Polynomial}.
In section \ref{Section Carleson}, we will establish the Carleson type estimates and prove the Theorem 
\ref{Main Frequency}.
In Section \ref{Section Single}, we will discuss the boundary Harnack principle on a single given nodal domain 
and then we prove Theorem \ref{Main Measure}.

\begin{remark}
    Although we only consider the elliptic operators in divergence form in this article, one could 
easily extend all results in this article to the elliptic operators in non-divergence form with 
Lipschitz continuous leading coefficients. It would be also interesting to obtain a parabolic counter 
part.
\end{remark}

{\bf Acknowledgements} The research of authors are partially supported by the NSF grant, 
DMS1955249.

\section{Preliminaries and Tools}\label{Section PRE}

Let $w$ be a $W^{1,2}$-solution of an elliptic equation in divergence form in $B_{10} \subset \mb{R}^n$ (the Euclidean ball with radius equaling to $10$ and center at $0$),
    \begin{equation}\label{Divergence Equation}
        \mc{L}(w) \equiv \mr{div}(A(x)\nabla w(x)) = 0 , 
    \end{equation}
where the symmetric matrix-valued function, $ A(x) = (a_{ij})_{n\times n}$, satisfies
    \begin{equation}\label{Coefficient 1}
        \ \lambda \cdot \mr{I} \leq A \leq \lambda^{-1} \cdot  \mr{I}\ ,
    \end{equation}
with Lipschitz entries, 
    \begin{equation}\label{Coefficient 2}
       \|a_{ij} \|_{\mr{Lip}} \leq \Lambda_1,
    \end{equation}
for some positive constants $\lambda$ and $\Lambda_1$. In this paper, we write $\mc{L}, \mc{L}_1$ etc. for elliptic
operators which satisfy the conditions (\ref{Divergence Equation}) (\ref{Coefficient 1}) and (\ref{Coefficient 2}). We use $L, L_1$ etc. to denote uniformly elliptic operators that satisfy only (\ref{Divergence Equation}) and (\ref{Coefficient 1}). 
For simplicity, we will use the
notation $C(\Lambda)$ to denote positive constants which depend only on $\lambda$, 
$\Lambda_1$ and $n$ and call them universal constants. And we shall use $C(\lambda)$ for 
constants depending only on the (\ref{Coefficient 1}) and the dimension $n$. Most of 
constants appeared in the paper will depend only on the dimension $n$, the ellipticity 
constant $\lambda$ and the doubling constant $N_0$ for solutions of such uniformly elliptic 
operators $L$.
By the standard interior estimates, if $\mc{L}(w)=0$ in $B_{10}$, then $w$ is in $C^{1,\alpha}(B_9)$ for any $\alpha \in (0,1)$. For general uniformly elliptic operators $L$, one has $w$ is in $C^{\alpha}$ for some positive $\alpha$ by De Giorgi's theorem, see \cite{HL11}. We define $Z(w) \equiv \{x\in B_{10} | \ w(x) =0 \}$ as the zero set of $w$ in $B_{10}$. For any point $x \in B_9$, we also define $\delta_w (x) \equiv \dist(x,Z(w))$ and use $\delta(x)$ if there is no ambiguity.

\subsection{Frequency Function and Doubling Index}\label{Introduction to Frequency}
Let us first recall the frequency function, which goes back to works of Agmon \cite{A66} and Almgren \cite{A79}, and was further developed 
in \cite{GL86}, see also \cite{L91}. This is a useful ingredient in estimating the size of nodal sets and the size of critical sets.
We refer to \cite{LM19} for more recent developments with much improved sharp results 
and other applications of the frequency functions. For the convenience we recall and collect 
a few basic facts about the frequency function and its important consequences.

The frequency function for a solution $w$ of $\mc{L}(w) = 0$ is defined as 
    \begin{equation}
        N_w(B_r(x_0)) =  N_w (x_0,r) = \frac{r \cdot \int_{B_r(x_0)} \lag A(x)\nabla w , \nabla w \rag }{\int_{\pa B_r(x_0)} \mu(x) |w|^2} \ ,
    \end{equation}
and we also set
    \begin{equation}
        H_w(x_0,r) = r^{1-n} \int_{\pa B_r(x_0)} \mu(x) |w|^2 ,
    \end{equation}
where $\mu(x) = \lag A(x)x,x \rag /|x|^2$, $0<\lambda < \mu(x) < n \cdot \lambda^{-1}, 
x_0 \in 
B_9$. Here $B_r(x_0)$ is the Euclidean ball with radius equaling to $r$ and center at $x_0$. If $x_0 = 0$, we use $B_r = B_r(0)$. 
And if there is no ambiguity, we often use $N(r)$ and $H(r)$ (or $N_w(r)$ and $H_w (r)$) for simplicities. By \cite{GL86}, one has
    \begin{equation}\label{Derivative of H}
        \frac{H'}{H} = \frac{2 N(r)}{r} + O(1),
    \end{equation}
and $O(1)$ is bounded by a universal constant $C_1 = C_1(\Lambda)$. We then have the following monotonicity theorem from \cite{GL86}.

\begin{theorem}\label{Monotone Frequency}
    \textit{
    There is a positive constant $C_2 = C_2(\Lambda)$, such that $\exp(C_2r) \cdot N(r)$ is a nondecreasing function of $r$.
    }
\end{theorem}

A main consequence of Theorem \ref{Monotone Frequency} is the doubling estimate. By 
using (\ref{Derivative of H}), one has
    \begin{equation}\label{Doubling}
       \begin{split}
        \log\bigg(\frac{H(2R)}{H(R)}\bigg) &\leq \int_R ^{2R} 2 \exp(-C_2 r)\frac{\exp(C_2 r) N(r)}{r} \ \mr{dr} + C_1 R
            \\ &\leq C(\Lambda) \cdot N(2) + C_1(\Lambda) 
       \end{split}
    \end{equation}
for any $x_0 \in B_8$ and $R <1$. For $|\nabla w|$, we have a similar doubling 
estimate, which is also derived in \cite{GL86}.

\begin{theorem}\label{Doubling Condition}
    \textit{
    Assume that $w(0) = 0$ and $N_w(B_5) \leq N_0 < \infty$. Then for any $x \in B_2$ $R \in (0,1)$, we have
    \begin{equation}\label{Double Function}
        \int_{B_{2R}(x) } |w|^2 \leq 2^{K_1 N_0}\int_{B_{R}(x)} |w| ^2
    \end{equation}
    \begin{equation}\label{Double Gradient}
        \int_{B_{2R}(x) } |\nabla w|^2 \leq 2^{K_2 N_0}\int_{B_{R}(x)} |\nabla w| ^2
    \end{equation}
    for some universal constants $K_1$ and $K_2$.
    }
\end{theorem}

One can then easily derive the following versions of three spheres theorems.
\begin{theorem}\label{Three Sphere}
    \textit{
    There exist universal constants $K_3, K_4$ and universal constants $\alpha_1,\alpha_2 \in (0,1)$, such that for any $x \in B_1$,
        \begin{equation}\label{Three 1}
            \sup_{B_1(0)} |w| \leq K_3 \sup_{B_{1/8}(x)}|w|^{\alpha_1} \cdot \sup_{B_{2}(0)}|w|^{1-\alpha_1} \ , 
        \end{equation}
        \begin{equation}\label{Three 2}
            \sup_{B_1(0)} |\nabla w| \leq K_4 \sup_{B_{1/8}(x)}|\nabla w|^{\alpha_2} \cdot \sup_{B_{2}(0)}|\nabla w|^{1-\alpha_2} \ .
        \end{equation}
    }
\end{theorem}
Note that (\ref{Three 1}) is a consequence of (\ref{Double Function}) and one can get 
(\ref{Three 2}) by (\ref{Double Gradient}) (or by (\ref{Double Function}), the 
Caccioppoli estimate and the Poincar\'{e} inequality) in a similar way. Recently, in \cite{LM18}, Logunov and 
Malinnikova have improved substantially (\ref{Three 1})
and (\ref{Three 2}) by establishing a sharp Remez type estimates for solutions.

Before we proceed further, we want to point out the following equivalence of norms:
\begin{lemma}\label{Equivalence of Norm}
    \textit{
    There are universal constants $c_1, c_2$ such that for any $0<r<4$,
        \begin{equation}
            \sup\limits_{B_r(x_0)}|w|^2 \leq c_1 \fint_{B_{3r/2}(x_0)} |w|^2 \leq c_2 \cdot H(x_0,2r),
        \end{equation}
    for any $x_0 \in B_6$.
    }
\end{lemma}
The first inequality follows from the De Giorgi's theorem \cite{HL11}, while the second 
inequality follows also from (\ref{Derivative of H}) and $N(r) \geq 0$, and it is a 
general fact for subsolutions.

In the following, we will use $\sup | \cdot |$ norm for most of estimates.
First as in \cite{LM19} one defines the doubling index
    \begin{equation}
        N_D(w,B_r (x)) = \log\bigg(\frac{\sup_{B_{r}(x) }|w|}{\sup_{B_{r/2}(x)} |w|}\bigg),
    \end{equation}
for any $B_{r}(x) \subset B_9$. Or more generally, one defines 
    \begin{equation}
        \widetilde{N}_D (w,B_r (x)) = \sup_{B_s(y) \subset B_r(x)}\log\bigg(\frac{\sup_{B_{s}(y) }|w|}{\sup_{B_{s/2}(y)} |w|}\bigg).
    \end{equation}
The doubling index $N_D(r)$ and frequency function $N(r)$ are equivalent because of Lemma \ref{Equivalence of Norm}, and we have the following inequalities:
    \begin{equation}\label{Equivalence Frequency 1}
        K_1 ^{-1} N(r/2) - K_2 \leq N_D(r) \leq K_1 N(2r) + K_2
    \end{equation}
and
    \begin{equation}\label{Equivalence Frequency 2}
        K_1 ^{-1} \widetilde{N}_D(r/2) - K_2 \leq N_D(r) \leq \widetilde{N}_D(r),
    \end{equation}
for some universal constants $K_1, K_2$ and all $r \in (0,8)$. See \cite{GL86}, \cite{L91} and \cite{LM19} 
for details. Hence forth,  without ambiguity, for either $L$ or $\mc{L}$, when we say $N_w \leq N_0$, we mean that  $\widetilde{N}_D (w,B_8(0))$ is bounded by $N_0$. And for $\mc{L}$, we will always, doubling the size of balls if necessary, use the equivalence of $N(r)$, $N_D(r)$ and $\widetilde{N}_D(r)$.

Finally, let us give another application of these statements above. It is a
growth estimate of $|w(x)|$ in terms of $\delta(x)$ for $x$ near the nodal set 
of the solution $w$, which will be an important ingredient in our paper.

\begin{theorem}\label{Growth Estimate}
    \textit{
    Suppose that $L(w) = 0$ in $B_{10}$ for $L$ only satisfying (\ref{Coefficient 1}), with $Z(w) \cap B_4 \neq \emptyset$ and $N_w \leq N_0 < \infty$, then there exist positive constants $A_1(\lambda),A_2(\lambda,N_0)$ and $\alpha(\lambda) \in (0,1)$, such that
        \begin{equation}
            A_1 \cdot \sup_{B_8}|w| \cdot \dist^{\alpha}(x,Z(w)) \geq |w(x)| \geq A_2 \cdot \sup_{B_8}|w| \cdot \dist(x,Z(w))^{N_0}
        \end{equation}
    for any $x\in B_{2}$.
    }
\end{theorem}
\begin{proof}
    We can assume that $\sup_{B_8}|w| =1 $. The left side inequality follows directly 
from the De Giorgi theorem, see \cite{HL11}. For the right side one, let $ r = \dist(x,Z(w))$, 
then the usual
Harnack inequality implies that $\sup_{B_{r/2}(x)} |w|\leq h(\lambda) \cdot |w(x)|$ for some $h(\lambda) >1$.
By the definition of $N_w \leq N_0$, we know that $2^{-k N_0} \cdot \sup_{B_{2^{k-1} \cdot r}(x)} |w| \leq \sup_{B_{r/2}(x)} |w|$ for all $k \in \mb{Z}_+$, which yields the conclusion. 
\end{proof}

\begin{remark}
    One can easily find scaled versions of the above growth estimate on balls of size 
$r$. For operators with analytic coefficients (hence solutions are also analytic in 
the interior), the above growth estimate can be derived from the \L ojasiewicz inequality as 
in \cite{LM16}.  However, all the constants involved will depend on the real analytic nature of the variety $Z(w)$. It is thus not so convenient to obtain uniform estimates when the nodal sets $Z(w)$ or operators involved are perturbed. If the coefficients are Lipschitz continuous, the gradients of the solutions $w$ satisfy the same growth estimates, see \cite{GL86}.
\end{remark}

\subsection{A Compact Class of Solutions}\label{Introduction to Compact Family}
Our second tool builds on the compactness of a class of solutions to any elliptic equations satisfying (\ref{Coefficient 1}) and (\ref{Coefficient 2}), which are defined as follows:
    \begin{equation}\label{Definition of Compact Family}
        \mc{S}_{N_0} (\Lambda) \equiv\{ w \in W^{1,2} \ | \ \mc{L}(w) = 0 \ \mr{in} \ B_{10},\ \mc{L} \ \mr{satisfies} \ (\ref{Coefficient 1}),(\ref{Coefficient 2}) , \ N_w \leq N_0, \ \sup_{B_8}|w| = 1\}.
    \end{equation}
This is a compact family in local $C^{1,\alpha}$-metric. A 
direct consequence is the compactness of their zero sets, i.e., 
    \begin{equation}
        \mc{F}_{N_0} (\Lambda) = \{ Z(w) \cap \overline{B_8} \ | \ w \in \mc{S}_{N_0}(\Lambda) \}
    \end{equation}
is compact under the Hausdorff distance.

The class $\mc{S}_{N_0}(\Lambda)$ is usually used to give upper bounds for the size of nodal 
sets or the size of critical sets. Let's summarize these estimates into the following 
statements:
    \begin{equation}\label{Measure Zero Set}
        H^{n-1} \big(Z(w) \cap B_4\big) \leq P_1 (\Lambda, N_0)
    \end{equation}
and 
    \begin{equation}\label{Measure Singular Set}
        H^{n-2} \big(S(w) \cap B_4\big) \leq P_2 (\Lambda, N_0)
    \end{equation}
for any $w \in \mc{S}_{N_0} (\Lambda) $. Here, $S(w) \equiv \{x \in B_9 \ | \ w(x) = |\nabla w|(x)  =0\}$ and the two positive constants $P_1,P_2$ only depend on $\Lambda$ and $N_0$. 

There are sveral important contributions for these two estimates, see for 
examples, \cite{DF88}, \cite{HS89}, \cite{L91}, \cite{HHL98} and \cite{CNV15}.
The best estimates upto date are $P_1 = M_1(\Lambda) \cdot N_0 ^{\alpha}$ for some $\alpha = \alpha(n) >1$ and $P_2 =\exp(M_2 (\Lambda) \cdot 
N_0 ^2)$, which are in \cite{LM18}, \cite{L18} and \cite{NV17} separately. 
It is worth to point out that in \cite{CNV15} and \cite{NV17}, the authors also 
established estimates on the Minkovski content, that is, the volume of a small 
neighborhood of $Z(w)$ and $S(w)$. Moreover, the set $S(w)$ can be replaced by $C(w)$,
 the set of all points $x\in B_9$ such that $ |\nabla w(x)| = 0$, see for examples 
\cite{HL00} and \cite{CNV15}.

\section{Modified Harnack Chain and Corkscrew Condition}\label{Section Corkscrew}
In this section, we will show some geometric properties of nodal domains. Surprisingly, some of them are similar to properties of NTA domains
\cite{JK82}, which have been influential in potential analysis on non-smooth domains and which have applications to many problems including the 
regularity of free boundaries.  For a domain to be NTA, it needs to satisfy two assumptions called the corkscrew condition and the Harnack 
chain condition. It is not hard to find examples of nodal domains that are not NTA. In some sense, typical nodal domains are like Lipschitz cones at 
sufficiently small scales and at larger scales they are more like twisted H\"{o}lder domains with complicated topology. For the class of uniformly elliptic 
operators with bounded measurable coefficients, so long as the solutions that are considered 
satisfy this additional doubling property (\ref{Double Function}) , the associated nodal domains will satisfy 
a corkscrew condition and a modified Harnack chain condition. In the (uniformly) analytic 
case, it is proved in \cite{LM16}. Our proof of the following statement is a generalization 
of that in \cite{LM16}. It builds on the natural scaling invariant property for this class 
of nodal domains.

\begin{theorem}\label{Stubborn Ball}
    \textit{
Suppose that $L(w) = 0$ in $B_{10}$, with $0 \in Z(w)$ and $N_w \leq N_0 <\infty$. Then, for any nodal domain $\Omega$ of $w$ with $ \Omega \cap B_1 \neq \emptyset$, and any $x \in \Omega \cap B_1$, there is a chain of points $\{x_i\}_{i=0} ^m \subset \Omega$ with $x_0 = x$ and satisfying the following properties: for $i = 0,1, \dots, m-1$,
    \begin{itemize}
        \item [(1)\label{Modified Harnack}] (Modified Harnack Chain.) 
            \begin{itemize}
                \item [(i)] $|w(x_{i+1})| \geq C_3(\lambda,N_0)|w(x_i)|$ for some $C_3>1$;
                \item [(ii)] $|x_{i+1} - x_i| \leq \big(1-\theta_0(\lambda,N_0)\big) \cdot \delta(x_i)$ for some $\theta_0 \in (0,1)$, $x_i \in B_{2}$ and $\delta(x_i) \leq 1/4$;
                \item [(iii)] $x_m \in B_{3} \backslash B_{2}$ or $x_m \in B_{2}$ but $\delta(x_m) > 1/4$;
                \item [(iv)] $m \leq - \xi_1(\lambda,N_0) \log(\delta(x_0)) + \xi_2(\lambda,N_0)$ for some $\xi_1,\xi_2 >0$.
            \end{itemize}  
        \item [(2)\label{Corkscrew}] (Corkscrew Condition.) $\delta(x_m) > c_4(\lambda, N_0) $ for some $c_4 \in (0,1/4)$, and hence $B_4 \cap \Omega$ contains a ball of radius $c_4/2$.
    \end{itemize} 
    }
\end{theorem}

If one considers all nodal domains of $w$ that intersect with $B_1$, the second statement in 
the above theorem 
implies exactly the two-sided corkscrew condition as in the definition of NTA domains.

The first statement in the above theorem leads to modified Harnack chains. 
One does have that the values of $w(x_i)$ grow geometrically. But it implies only that 
$x_i$'s stay away from $Z(w)$ (in a same nodal domain) by a power of its distance 
to the boundary of the nodal domain.  This latter geometric picture is consistent with the Theorem \ref{Growth Estimate}. In this connection, we find that there is an 
interesting connection to the hyperbolic metric defined on the nodal domains, which is
the Euclidean metric multiplying by the conformal factor $ w^{-2}$. But we shall not explore it
in this article.

\begin{lemma}\label{Enlarge Chain Lemma2}
    \textit{
    Suppose that $L(w) = 0$ in $B_{10}$, with $0 \in Z(w)$ and $N_w \leq N_0 < \infty$, then there are constants $C_3 = C_3(\lambda,N_0) >1$ and $\theta_0 = \theta_0(\lambda,N_0) \in (0,1)$, 
such that for any $x\in B_{2}$ with $w(x) \neq 0$ and $\delta(x)  \leq 1/4$, there is an $\tilde{x}\in B_{3}$ with $|x-\tilde{x}| 
\leq (1-\theta_0) \cdot  \delta(x) $ and $|w(\tilde{x})| > C_3|w(x)|$.
    }
\end{lemma}
\begin{proof}
    Suppose that $w(x)>0$ and let $\delta = \delta(x)$. Set $\epsilon \equiv( 
\sup_{B_{(1-\theta)\delta}(x)} w)/ w(x) -1 >0$ with a positive and small $\theta$ to be 
chosen later. Since $L(w(\cdot)-w(x)) = 0$, by the usual Harnack's inequality,
        \begin{equation}\label{Enlarge Chain Size 1}
            \sup_{B_{(1-2\theta)\delta}(x)} | w(\cdot) - w(x)| 
            \leq  C(\lambda,\theta) \cdot \sup_{B_{(1-\theta)\delta}(x)} \big( w(\cdot) - w(x) \big) = C(\lambda, \theta) \cdot \epsilon w(x).   
        \end{equation}
    By the definition of $N_w \leq N_0$ and the usual Harnack's inequality, we have that
        \begin{equation}
            \sup_{B_{2\delta}(x)}| w| \leq 4^{ N_0}  \sup_{B_{\delta/2}(x)}| w| \leq 4^{N_0} \cdot  C(\lambda)  \cdot w(x) .
        \end{equation}
    On the other hand, since there is an $x_* \in Z(w)$ such that $|x-x_*| = \delta$, the De Giorgi's theorem yields that
        \begin{equation}\label{Enlarge Chain Size 2}
            \sup_{B_{4 \theta \delta}(x_*)} |w| \leq C(\lambda) \cdot \theta^{\alpha} \sup_{B_{2\delta}(x)} |w| \leq C(\lambda)  4^{N_0} \cdot \theta^{\alpha} \cdot  w(x)
        \end{equation}
    for some $\alpha = \alpha(\lambda)$ and for every $\theta \in (0,1/16)$. Now we choose a $\theta = \theta(\lambda,N_0) \in (0,1)$ such that $C(\lambda) 4^{N_0} \cdot \theta ^{\alpha} < 1/2$ in (\ref{Enlarge Chain Size 2}).
    
    Then, for any $y \in B_{4 \theta \delta }(x_*) \cap B_{(1-2\theta) \delta}(x)$, by (\ref{Enlarge Chain Size 1}) and (\ref{Enlarge Chain Size 2}), we have that 
        \begin{equation}
            \big(1- C(\lambda, \theta) \epsilon \big) \cdot w(x)\leq w(y) \leq \frac{1}{2} \cdot w(x),
        \end{equation}
    which yields that $\epsilon \geq c >0$ for some positive $c = c(\lambda, \theta) = c(\lambda ,N_0)$.
    
    Let $\theta_0 = \theta$, $C_3 = C_3(\lambda,N_0) = \epsilon +1$ and $\tilde{x}$ is a 
point on $\pa B_{(1-\theta)\delta}(x)$ which takes the value of $\sup_{B_{(1-\theta)\delta} (x)} w$.
   
\end{proof}

With Lemma \ref{Enlarge Chain Lemma2}, we can proceed to the proof of Theorem 
\ref{Stubborn Ball}.
\begin{proof}[Proof of Theorem \ref{Stubborn Ball}]
    For (i) and (ii), one simply applies Lemma \ref{Enlarge Chain Lemma2} iteratively. 
This iteration that satisfies both (i) and (ii) has to end after finitely many steps. 
We let $m$ and the corresponding $x_m$ as the first time that (iii) of Theorem \ref{Stubborn 
Ball} is valid.

    For (iv), the upper bound of $m$, by (i) and Theorem \ref{Growth Estimate}, we have that
        \begin{equation}
           (C_3) ^m \cdot A_2 \delta(x)^{N_0} \cdot \sup_{B_8}|w|\leq (C_3) ^m |w(x)| \leq |w(x_m)| \leq \sup_{B_8}|w|,
        \end{equation}
    which is equivalent to 
        \begin{equation}\label{Length of Harnack Chain}
            m \leq (- N_0\log(\delta(x)) - \log(A_2))/\log(C_3).
        \end{equation}
    Since $C_3 >1$, we get the desired $\xi_1$ and $\xi_2$.

    For the Corkscrew Condition, we first assume that $\delta(x_m)\leq 1/4$ and $\sup_{B_8}|w|=1$. From Theorem \ref{Growth Estimate}, there are $A_1(\lambda), A_2(\lambda,N_0) >0$ such that
        \begin{equation}
             A_1 \cdot \delta(y)^{\alpha} \geq |w(y)| \geq  A_2 \cdot \delta(y)^{N_0}
        \end{equation}
    for any $y\in B_{2}$. Hence, it suffices to show that $|w(x_m)| \geq 
C(\lambda,N_0)>0$. Because
        \begin{equation}
         |x_0 - x_m| \leq \sum_{i=0} ^{m-1} |x_i - x_{i+1}| 
        \end{equation}
    and 
        \begin{equation}
            |x_i - x_{i+1}| \leq \delta(x_i)\leq A_2 
^{-1/N_0}|w(x_i)|^{1/N_0}\leq 
A_2 ^{-1/N_0}|w(x_m)|^{1/N_0} \cdot C_3 ^{(i-m)/N_0},
        \end{equation}
    also, $|x_0 - x_m| \geq 2 - 1 = 1$ and $\sum_{i=0} ^{m-1} C_3 
^{(i-m)/N_0}$ 
is bounded by $1/(C_3^{1/N_0} - 1)$, we get a desired lower bound for 
$|w(x_m)|$.

\end{proof}

A direct corollary of the Corkscrew Condition is the local boundedness of number of nodal domains.

\begin{corollary}\label{Number of Domains}
    \textit{
    Suppose that $L(w) = 0$ in $B_{10}$, with $0 \in Z(w)$ and $N_w \leq N_0 <\infty$. Then, the number of nodal domains in $B_4$ which have nonempty intersections with $B_1$ is bounded by a positive integer $T_0 = T_0(\lambda,N_0)$.
    }
\end{corollary}

\section{Boundary Harnack, H\"{o}lder Continuity and Entire Solutions}\label{Section BDD}

We will use the corkscrew property of nodal domains to provide versions of boundary Harnack principle. We observe first the following:

\begin{lemma}
    \textit{
    Assume that $\mc{L}u = \mc{L}v = 0$ in $B_{10}$, with $\sup_{B_8} |u| = \sup_{B_8}|v| = 1$, and $N_D(u,B_8) \leq N_0 <\infty$. If $1 <2^{N_0 +1}<m < \infty$, then $N_D(mu-v, B_8) \leq N_0 + 2 < \infty$.
    }
\end{lemma}
\begin{proof}
    \begin{equation}
        \frac{\sup_{B_8}|mu-v|}{\sup_{B_{4}} | mu-v|} \leq \frac{m +1}{m \cdot \sup_{B_{4}}|u| - 1} \leq \frac{m+1}{m \cdot 2^{-N_0} -1} \leq 2 \frac{m+1}{m \cdot 2^{-N_0}} \leq 4 \cdot 2^{N_0}
    \end{equation}
\end{proof}

\begin{theorem}\label{Local Boundedness}
    \textit{
    Suppose that $\mc{L}u=\mc{L}v = 0$ in $B_{10}$, and $0 \in Z(u) \subset Z(v)$, with $\sup_{B_8} |u| = \sup_{B_8}|v| = 1$. If $N_u \leq N_0 < \infty$, then there is a positive constant $C = C(\Lambda,N_0) < \infty$ such that $|v/u| \leq C $ in $B_{1} \backslash Z(u)$ .
    }
\end{theorem}
\begin{proof}
    First, we show that there is a large $C = C(\Lambda ,N_0)$ such that $Cu-v$ has the same nodal domains as $u$ in 
$B_{1}$. Set $\delta(x) \equiv \delta_{Z(u)}(x)$.
    
    Let $E =\{x \in B_{3} \ | \ (Cu-v)(x) \cdot u(x) > 0\ \}$. We can first assume that $C > 2^{N_0+1}$ 
as in the previous lemma and then  by (\ref{Equivalence Frequency 2}) get that $N_{Cu-v} \leq \mc{N}_0 \equiv K_1(\Lambda)(N_0+2) + K_2(\Lambda)$ for some $K_1,K_2 >0$. Let $ E_1 = \{x \in B_3 \ | \  \delta(x)>c_4 / 
8\}$, where $c_4 = c_4(\mc{N}_0,\Lambda)$ is the same constant appeared in the Corkscrew Condition of 
Theorem \ref{Stubborn Ball} for $Cu-v$. By Theorem \ref{Growth Estimate}, we have $|u(x)| > A_2 
(c_4/8)^{N_0} \equiv c$ for any $x \in E_1$. 
    Let us fix $C = 2\max\{c^{-1} , 2^{N_0 +1} \} $. For this $C$, we have $E_1 \subset E$ because for 
any $x \in E_1$, we have $|Cu(x)|>2$, and then $(Cu-v)(x)$ and $u(x)$ must have the same sign.
    
    For this fixed $C$, assume that there is an $x \in B_{1}$ such that $u(x) > 0$ but $(Cu-v)(x) 
\leq 0$. Note that if $(Cu-v)(x) =0$, by the strong maximum principle and unique continuation, we can 
always choose another point $y$ arbitrarily close to $x$ with $(Cu-v)(y) <0$ but $u(y) >0$. So , we 
assume that $(Cu-v)(x) <0$.
    Therefore, this $x$ is in a negative nodal domain $\Omega$ of $Cu-v$ in $B_3$, which means that $Cu-v<0$ 
in $\Omega$ and $\Omega$ is in the complement of $E$. On the other hand,  since $x$ is in a positive nodal domain 
of $u$ in $B_3$, which we 
denote as $\Omega_1$,  $\Omega$ is contained in $\Omega_1$ and is certainly not connected with other 
nodal domains of $u$. So, $\Omega \subset \Omega_1 \backslash E $.
    By the corkscrew property as in the Theorem \ref{Stubborn Ball} for $Cu-v$ (note the doubling 
index is bounded by $N_0 + 2$ independent of large $C$) , there is a point $x_m \in \Omega \cap 
B_3$ such that $\delta(x_m) \geq \dist(x_m,Z(Cu-v))  > c_4$, which is clearly impossible by
our construction of $E_1$ and the fact that $E_1 \subset E$.
    Hence we have proved that $B_{1}\backslash Z(u) \subset E$, which means that $Cu-v$ has the same 
nodal domains as $u$ in $B_{1}$.

    Similarly, for the same $C$, we can show that $C u + v$ has the same nodal domains as $u$ in $B_{1}$. Hence, $|v/u| \leq C$ in $B_{1} \backslash Z(u)$.
\end{proof}

\begin{corollary}\label{Two Sides Bounds}
    \textit{
    Suppose that $\mc{L}u=\mc{L}v =0$ in $B_{10}$, $0 \in Z(u) = Z(v)$, and $\sup_{B_8}|u| = \sup_{B_8} |v| = 1$, with $N_u\leq N_0$ and $N_v \leq N_0$ for some positive $N_0 < \infty$. Then, there is a positive constant $C = C(\Lambda, N_0) < \infty$, such that $C^{-1} \leq |v/u| \leq C$ in $B_{1} \backslash Z(u)$.
    }
\end{corollary}
\begin{proof}
    Switch the position of $u$ and $v$ in Theorem \ref{Local Boundedness}.
\end{proof}

\begin{remark}
We should note that both the Theorem \ref{Local Boundedness} and the Corollary \ref{Two Sides Bounds} remains true when $\mc{L}$ is replaced by $L$, see
Theorem \ref{Boundary Harnack in Single Domain} in section \ref{Section Single}. On the other hand, by Theorem \ref{Bound Frequency} that we will prove in the next section, we can drop the assumption 
that $N_v \leq N_0$ because Theorem \ref{Bound Frequency} implies that $N_v \leq 
D(\Lambda,N_0) < \infty$ on $B_1$. And consequently one can prove the boundedness of $|v/u|$ on $B_{1/10}\backslash Z(u)$ as in Theorem \ref{Local Boundedness}.
\end{remark}

Next, we show that $v/u$ satisfies a strong maximum principle, which was noted in the Remark 
2.8 of \cite{LM15} for the case $\mc{L} = \Delta$.

\begin{theorem}\label{Strong Maximum Principle}
    \textit{
    Suppose that $\mc{L}u = \mc{L}v = 0$ in $B_{10}$, and $Z(u) \subset Z(v)$, then $\sup_{B_8} v/u$ cannot be achieved at $x \in B_8$ if $v/u$ is not a constant.
    }
\end{theorem}
\begin{proof}
    Denote $\sup_{B_8} v/u$ as $M$, we consider $Mu-v$. We can assume that $Mu -v \not\equiv 0$. For $x \in B_8$, if $u(x) >0$, then $M\geq v(x)/u(x)$ and then $Mu(x) - v(x) \geq 0$. By the usual strong maximum principle, we know $Mu(x) - v(x) >0$.
    Similarly, if $u(x) <0$, we have $Mu(x) - v(x) <0$. Hence, $M$ is not achieved at $x \in B_8 \backslash Z(u)$. These also tell us that $Z(u) \cap B_8 = Z(Mu-v) \cap B_8$. 

    Now, for any $x_0 \in Z(u) \cap B_8$, 
consider $B_{10r}(x_0)$ for some $r$ small enough. By Theorem \ref{Local Boundedness},
        \begin{equation}
            \begin{split}
                \inf_{B_{r}(x_0) \backslash Z(u)} \bigg| M -\frac{v}{u}\bigg| &= \inf_{B_{r}(x_0) \backslash Z(u)} \bigg| \frac{Mu-v}{u}\bigg| 
                \\  &= \frac{1}{\sup_{B_{r}(x_0) \backslash Z(u)} |u/(Mu-v)|}
                \\  &\geq C \cdot \frac{\sup_{B_{8r}(x_0)}|Mu-v|}{\sup_{B_{8r}(x_0)}|u|} > 0,
            \end{split}
        \end{equation}
    where $C$ is a positive constant depending on $\Lambda$ and $\widetilde{N}_D (Mu-v, B_{10r}(x_0)) < 
\infty$. Hence, we conclude that in $B_{r}(x_0)$, $M > v/u$. And then $M > (v/u)(x)$ for any $x$ 
strictly inside $B_8$. 
\end{proof}

To end this section, we are going to work on the continuity of $v/u$. If we only need the continuity of $v/u$ at some point $x_0 \in Z(u)$, we can consider $u$ and $v-(v/u)(x_0) \cdot u$ in Theorem \ref{Local Boundedness} and use Taylor expansion of $u, v$ at $x_0$. See \cite{H94} for the Taylor expansion and \cite{LM15} for more in harmonic functions cases. But in this way, the continuity scale will depend on the point $x_0$. 
In $\mb{R}^2$ case, this way will also give differentiability of $v/u$ since the formal gradient of $v/u$ at $x_0 \in S(u) = \{x \ | \ u(x) = |\nabla u|=0\}$ is $0$.

Here, we are going to show the H\"{o}lder continuity of $v/u$, and the proof is quite standard if we 
apply also the conclusion of Theorem \ref{Bound Frequency} which will be proven in the next 
section.

\begin{theorem}\label{Holder Continuity}
    \textit{
    Suppose that $\mc{L}u=\mc{L}v = 0$ in $B_{10}$, $N_u \leq N_0 < \infty$, and $0 \in Z(u) \subset Z(v)$, then $v/u \in C^\alpha(B_{1/10})$ for some $\alpha = \alpha(\Lambda,N_0) \in (0,1)$.
    }
\end{theorem}
\begin{proof}
    We are going to show the oscillation decay estimate at $0$. If 
        \begin{equation}
            \sup_{B_{1/100}}\frac{v}{u} \leq \frac{1}{2} (\sup_{B_{1}} \frac{v}{u} + \inf_{B_{1}} \frac{v}{u}),
        \end{equation}
    then
        \begin{equation}
            \sup_{B_{1/100}}\frac{v}{u} - \inf_{B_{1/100}}\frac{v}{u} \leq \frac{1}{2} \bigg(\sup_{B_{1}} \frac{v}{u} - \inf_{B_{1}} \frac{v}{u} \bigg).
        \end{equation}
    If 
        \begin{equation}
            \sup_{B_{1/100}}\frac{v}{u} \geq\frac{1}{2} (\sup_{B_{1}} \frac{v}{u} + \inf_{B_{1}} \frac{v}{u}),
        \end{equation}
    we consider $v^* (x) = (v - (\inf_{B_{1}}(v/u)) \cdot u)(x/10)$ , $u^* (x) = u(x/10)$. Note that $u^*$ and $v^*$ have the same zero set in $B_{10}$ by the proof of Theorem \ref{Strong Maximum Principle}, $ v^* u^*\geq 0$, and $N_{u^*} \leq N_{u} \leq N_0$. By Theorem \ref{Bound Frequency}, $N_{v^*} \leq D=D(\Lambda,N_0)$ in $B_{1}$. Then, by Corollary \ref{Two Sides Bounds}, with a larger constant $C = C(\Lambda,D) = C(\Lambda,N_0)$ in it, we can show, 
        \begin{equation}
            \inf_{B_{1/100}} \frac{v}{u} - \inf_{B_{1}} \frac{v}{u} = \inf_{B_{1/10}} \frac{v^*}{u^*} \geq C^{-2}\sup_{B_{1/10}} \frac{v^*}{u^*} \geq \frac{1}{2C^2} (\sup_{B_{1}} \frac{v}{u} - \inf_{B_{1}} \frac{v}{u}),
        \end{equation}
    and then
        \begin{equation}\label{Oscillation Decay}
            \sup_{B_{1/100}} \frac{v}{u} - \inf_{B_{1/100}} \frac{v}{u} \leq \bigg(1- \frac{1}{2C^2}\bigg) (\sup_{B_{1}} \frac{v}{u} - \inf_{B_{1}} \frac{v}{u}).
        \end{equation}
\end{proof}

A direct corollary of Theorem \ref{Holder Continuity} and (\ref{Oscillation Decay}) is the following Liouville theorem for the case that $\mc{L} = \Delta$. In this case, all the constants $C(\Lambda,N_0)$ will be replaced by $C(n,N_0)$ so that we can do both blow-ups and blow-downs. And all the theorems in this section are valid with constants of the form $C(n,N_0)$.

\begin{corollary}\label{Liouville Theorem}
    \textit{
Suppose that $\Delta(u) = \Delta(v) =0$ in $\mb{R}^n$, $N_u(0,r) < N_0 < \infty$ for all $ r >0$, and $Z(u) \subset Z(v)$. Then, there is a $\beta = \beta(n,N_0) \in (0,1)$, such that if 
 \begin{equation}\label{Liouville Growth}            \liminf_{r \to \infty} r^{-\beta} \cdot \sup_{B_r} \frac{v}{u} < \infty,
        \end{equation}
    we will have that $v = c \cdot u$ for some $c \in \mb{R}$. In particular, if $Z(u) = Z(v)$, the condition (\ref{Liouville Growth}) will be satisfied, and then there is a constant $c \in \mb{R}\backslash\{0\}$ such that $v = c \cdot u$.
    }
\end{corollary}
\begin{proof}
    If $Z(u) \neq Z(v)$, we may assume that $(v/u)(0) = 0$. Then, if $\sup_{B_{r}} | v/u | = - \inf_{B_r} v/u$, one can denote $M = \sup_{B_{100r}} v/u$ and consider $M - v/u = (Mu-v) / u$ on $B_{100r}$. Since $Mu-v$ and $u$ have the same zero set in $B_{100r}$, by Theorem \ref{Bound Frequency} and Corollary \ref{Two Sides Bounds},
    there is a constant $C = C(n,N_0) > 1$ such that
        \begin{equation}
            M + \sup_{B_r} \bigg|\frac{v}{u}\bigg| = \sup_{B_r} \frac{Mu - v}{ u } \leq C \cdot \inf_{B_r} \frac{Mu - v}{ u } = C \cdot M - \sup_{B_r} \frac{v}{u} \leq C \cdot M.
        \end{equation}
    Hence, there is always a constant $M_1 = M_1 (n,N_0) > 1$, such that 
        \begin{equation}
            \sup_{B_r} \bigg| \frac{v}{u} \bigg| \leq M_1 \cdot \sup_{B_{100r}} \frac{v}{u}.
        \end{equation}
    By the proof of Theorem \ref{Holder Continuity} and (\ref{Oscillation Decay}), there is a constant $\theta=\theta(n,N_0) \in (0,1)$ such that 
        \begin{equation}\label{Global Decay}
            \sup_{B_{r}} \frac{v}{u} - \inf_{B_{r}} \frac{v}{u} \leq \theta^k \cdot (\sup_{B_{100^k r}} \frac{v}{u} - \inf_{B_{100 ^k r}} \frac{v}{u}) \leq \theta^k \cdot (2M_1) \cdot \sup_{100^{k+1} r} \frac{v}{u}
        \end{equation}
    for all $r > 0$ and $k \in \mb{Z}_+$. By choosing $\beta = \beta(n, N_0) \in (0,1)$ such that $\theta \cdot 100^{\beta} <1$, the statement follows if we let $k \to \infty$ and then $r \to \infty$.

    If $Z(u) = Z(v)$, by Corollary \ref{Two Sides Bounds} and Theorem \ref{Bound Frequency}, we will have that 
        \begin{equation}\label{Global Bound}
            \sup_{B_r} \bigg| \frac{v}{u} \bigg| \leq C \cdot \inf_{B_r} \bigg| \frac{v}{u} \bigg| \leq C \cdot \bigg|\frac{v}{u} \bigg|(0)
        \end{equation}
    for some $C = C(n, N_0)>0$ and all $r >0 $. Denote the right hand side of (\ref{Global Bound}) as $M_2$. By the first inequality of (\ref{Global Decay}), we have that
        \begin{equation}
            \sup_{B_{r}} \frac{v}{u} - \inf_{B_{r}} \frac{v}{u} \leq \theta^k \cdot (\sup_{B_{100^k r}} \frac{v}{u} - \inf_{B_{100 ^k r}} \frac{v}{u}) \leq \theta^k \cdot M_2.
        \end{equation}
The statement follows if we let $k \to \infty$ and then $r \to \infty$. 
We note that the above proof involves only controls of growth of both $u$ and $v$ at infinity. If one uses the fact that the operator is the standard Laplacian then the hypothesis on $u$ implies that $u$ is a harmonic polynomial. If the ratio $v/u$ grows like a power of $r$, then $v$ is also a harmonic polynomial. The conclusions can also be derived directly by working polynomials and simple blow-downs.
\end{proof}

\section{Uniform Bound on Frequency Functions for Solutions with the Same Zero Set}\label{Section Carleson}

In this section, all elliptic operators $\mc{L}, \mc{L}_1, \mc{L}_0$ will satisfy the conditions (\ref{Coefficient 1}) and (\ref{Coefficient 2}). Our main result is the following theorem.

\begin{theorem}\label{Bound Frequency}
    \textit{
    Suppose that $\mc{L}(u) = \mc{L}_1(v) = 0$ in $B_{10}$, and $0 \in Z(v) = Z(u) = Z$. Also assume that $N_u \leq N_0 < \infty$. Then, there is a positive constant $D = D(\Lambda, N_0) < \infty$, such that
        \begin{equation}
            \log \bigg( \frac{\sup_{B_{1}}|v|}{ \sup_{B_{1/2}}|v|} \bigg)  \leq D. 
        \end{equation}
    }
\end{theorem}
In order to prove this theorem, we first need a Carleson type estimate, which is always a key ingredient for the boundary Harnack principle, see for 
examples \cite{CFMS81}, \cite{JK82} and \cite{LM16}. The proof for this Lemma \ref*{Caffarelli's Lemma} is inspired by \cite{CFMS81}.

\begin{lemma}\label{Caffarelli's Lemma}
    \textit{
    Suppose that $L(u) = 0$ in $B_{10}$, $0 \in Z(u) =  Z$, and $N_u \leq N_0 < \infty$. Assume that $\Omega$ is a nodal domain of $u$ in $B_3$ which satisfies $\Omega \cap B_{1/2} \neq \emptyset$. Then, if $ L_1(v) = 0$ in $\Omega$, $v>0$ in $\Omega$, and $v =0$ on $Z \cap \pa \Omega$, there exist constants $M= M (\lambda,N_0) >0$ and $c = c(\lambda,N_0) >0$, such that the following estimate holds:
    \begin{equation}
        \sup_{B_{1/2} \cap \Omega} v \leq M \sup\limits_{y\in B_{2} \cap \Omega , \ \delta(y) \geq c} v(y).
    \end{equation}
    In particular, if $L_1(v) = 0$ in $B_{10}$ and $Z(v) = Z$, then
    \begin{equation}
        \sup_{B_{1/2}} |v| \leq M \sup\limits_{y\in B_{2}, \ \delta(y) \geq c} |v|(y).
    \end{equation}
    }
\end{lemma}

\begin{proof}
    Take $c= c_4/2$ for the $c_4$ in the Corkscrew Condition of Theorem \ref{Stubborn Ball}. Assume that $\sup\{|v(y)|\ | \ y \in B_{2} \cap \Omega, 
\delta(y) \geq c \} = 1$, then we shall prove that $\sup\{|v(y)|\ | \ y \in B_{1/4} \cap \Omega \} \leq M$ for some $M = M(\lambda,N_0)$.
    
    First, we claim that for any $x \in B_{1} \cap \Omega$, there are $\alpha_1(\lambda, N_0) >0$, $\alpha_2(\lambda, N_0) >0$, such that 
        \begin{equation}\label{Control Boundary Distance}
            v(x) \leq \alpha_2 \cdot \delta(x)^{-\alpha_1}.
        \end{equation}
    This claim follows from (iv) of Theorem \ref{Stubborn Ball} by going backwards. Indeed, since the length of the modified Harnack chain associated 
to $x$ is bounded by $-\xi_1\log(\delta(x))+\xi_2$ and $\delta(x_m) \geq c_4$, if we apply the usual Harnack inequality along this modified Harnack 
Chain, we will get
        \begin{equation}
            v(x) \leq h^m \cdot v(x_m) \leq h^{-\xi_1\log(\delta(x))+\xi_2} \cdot 1
        \end{equation}
    for some $h = h(\lambda,\theta_0) = h(\lambda,N_0) >1$ which is the constant in the Harnack inequality for this class of elliptic operators.
    
    Next, we need the following standard elliptic estimate for subsolutions: If $L(w) \geq 
0$ in $B_2$, $w \geq 0$ 
in $B_2$, and $|\{x\in B_2 \ | \ w(x)=0 \}| \geq \epsilon >0$, then $\sup_{B_1} w \leq \theta \cdot \sup_{B_2} w$ for some $\theta = \theta(\lambda,\epsilon) \in (0,1)$.

    We follow now the same type arguments as in \cite{CFMS81}. Assume that for some $y_0 \in 
B_{1/2} \cap \Omega$ and $|v(y_0)| = M_0>1$, then one has $\delta(y_0) < c $. Consider the ball $B_{ 3 \delta(y_0)} (y_0)$, on which $v$ may be 
regarded as a nonnegative subsolution if we extend $v$ to be $0$ out of $\Omega$. By the Corkscrew Condition of Theorem \ref{Stubborn Ball}, $B_{3\delta(y_0)}(y_0) \backslash \Omega$ contains a ball of radius $\delta(y_0) \cdot r$ with some small $r = r(\lambda,N_0)>0$.
    Hence, by the above estimate for nonnegative subsolutions, there is an $y_1 \in B_{3 \delta(y_0)}(y_0) \cap \Omega$, such that 
$v(y_1) \geq 
\theta^{-1} v(y_0) = \theta^{-1} M_0$ for a $\theta = \theta(\lambda,r) = \theta(\lambda , N_0) 
\in (0,1)$. Consequently, $\delta(y_1) < c$ so long as $y_1$ is also in $B_2$.

    We can continue this process to find  $y_2, y_3 , \dots$ so long as they all stay inside $B_2$. 
Let us estimate $|y_0 - y_i|$ for $i \geq 0$. Note that our construction gives that $|y_i - 
y_{i+1}| \leq 3 \delta(y_i)$. By (\ref{Control Boundary Distance}), if $y_i \in B_1 \cap \Omega$, then
        \begin{equation}
            \delta(y_i) \leq \beta_1 \cdot v(y_i) ^{-\beta_2} \leq \beta_1\cdot \theta^{\beta_2 i} \cdot v(y_0) ^{-\beta_2} = \beta_1 \cdot \theta^{\beta_2 i} \cdot M_0 ^{-\beta_2}.
        \end{equation}
    for some $\beta_1 = \beta_1(\lambda,N_0)>0$ and $\beta_2 = \beta_2(\lambda,N_0)>0$. Since 
$\theta <1$, the last right hand side terms form a convergent geometric series, and we can sum 
all of them up for $i = 1,2, \dots$. 
    
    If $M_0^{\beta_2} \geq 30\beta_1/(1-\theta^{\beta_2})$, then $|y_0 - y_i| \leq 1/10$ for all $i \geq 0$, and then all $y_i$ will stay in $B_1 \cap \Omega$. This is a contradiction since $v(y_i) \geq \theta^{-i}M_0 \to \infty$ as $i \to \infty$.
\end{proof}

We can now proceed with the proof of Theorem \ref{Bound Frequency}. The 
strategy is quite simple. The 
first step is to use Lemma \ref{Caffarelli's Lemma} to push the point where the solution 
$v$ taking approximate maximum values away from the nodal set. Next, we shall apply the Harnack 
inequality along paths fully contained in a nodal domain of $v$ (or equivalently, a nodal domain of $u$), which 
connects points in a larger ball far away from the zero set to points where $v$ reaches approximate
maximums inside a smaller ball. The difficulty is to avoid neck-like tiny regions in the process of 
connecting these points so that it can be done in a quantitatively controlled  manner.

\begin{proof}[Proof of Theorem \ref{Bound Frequency}]
    We need to consider the following family:
        \begin{equation}
            \mc{S}_{N_0} (\Lambda) \equiv\{ w \ | \ \mc{L}(w) = 0 \ \mr{in} \ B_{10},\ \mc{L} \ \mr{satisfies} \ (\ref{Coefficient 1}),(\ref{Coefficient 2}) , \ N_w \leq N_0,\  \sup_{B_8}|w| = 1\},
        \end{equation}
    which is a compact family in local $C^{1,\alpha}$-metric.
    
    We can then prove the statement by contradiction. If the theorem fails, assume that $\{u_n\} 
\subset \mc{S}_{N_0}(\Lambda)$ with $\sup_{B_8}|u_n| = 1$, $0 \in Z(u_n) = Z_n$. And $v_n$ 
satisfies that $\mc{L}_n (v_n) = 0$ in $B_{10}$, $Z(v_n)= Z_n$, with
        \begin{equation}\label{Blow Up Frequency}
            \log \bigg( \frac{\sup_{B_{1}}|v_n|}{ \sup_{B_{1/2}}|v_n|} \bigg)  \to \infty. 
        \end{equation}
    By compactness of the class $\mc{S}_{N_0}(\Lambda)$, we can assume that $u_n \to u_0 \in 
\mc{S}_{N_0}(\Lambda)$. Note that $0 \in Z(u_0)$ since the convergence is in local $C^{1,\alpha}$ metric.

    Let $Z_0 = Z(u_0)$, we make a partition of $B_2\backslash Z_0$ in terms of nodal domains. 
Let us assume that $B_2 
\backslash Z_0 = \sqcup_{i=1} ^{T}  (\Omega_i \cap B_2)$, where $\Omega_i$, $i = 1, \dots, T $ are disjoint nodal domains of $u_0$ in $B_3$ such that $\Omega_i \cap B_2 \neq \emptyset$.  Note that $T \leq T_0  = T_0(\lambda,N_0)$ by Corollary \ref{Number of Domains}.
    If we divide $[3/2,2)$ into $[3/2 + (j-1)/(4T_0), 3/2 + j/ (4T_0))$, $j = 1, \dots, 4T_0$, there will exist a $j = j(Z_0)$, such that for each $\Omega_i$, if $\Omega_i \cap B_{3/2 + j / (4T_0)} \neq \emptyset$, then $\Omega_i \cap B_{3/2+(j-1)/4T_0} \neq \emptyset$. We denote the subset of subindices of these $\Omega_i$ as $I_0$. Hence, we can set $\eta = 3/2 + (j-1)/(4T_0) + 1/(8T_0) \in (3/2 , 2)$ and $\epsilon = 1/(100 T_0) \ll 1$.
    We focus on the ball $B_\eta$. The point here is that those $\Omega_i$ with $i\in I_0$ are path-connected, form a partition of $B_\eta$ and also 
have a nonempty intersection with $B_{\eta - 10 \epsilon}$.

    By using Lemma \ref{Caffarelli's Lemma} to push maximum points away from zero sets locally, one can show that, for some positive $M = 
M(\lambda,N_0)$ and $c(\lambda, N_0)$, we can get
    \begin{equation}\label{Expel Points Outward}
        \sup_{B_{\eta- 2\epsilon}} |v_n| \leq M \sup_{y \in B_{\eta- \epsilon} ,\ \delta_n(y) \geq c} |v_n|(y),
    \end{equation}
    where $\delta_n(y) = \dist(y,Z_n)$. Assume the maximal value of the right hand side of (\ref{Expel Points Outward}) is achieved by $y_n$. Note that when $n$ is large enough, $\{y \in B_{\eta- \epsilon} \ | \ \delta_n(y) \geq c\}$ will be contained in $\{y \in B_{\eta- \epsilon} \ | \ \delta_0(y) = \dist(y, Z_0) \geq c/2\}$. Hence, we can assume that $\delta_0(y_n) \geq c/2$.

    Because for each $i \in I_0$, $\Omega_i \cap B_{\eta - 10 \epsilon} \neq \emptyset$, by the Corkscrew Condition of Theorem \ref{Stubborn Ball}, 
there is a small ball of radius $r = r(\lambda, N_0)>0$ with center $x_i$ inside $\Omega_i \cap B_{\eta - 4 \epsilon}$. Let $d = \min\{c/4 , r/2\}$, 
there is a constant $\mu = \mu(d,Z_0)>0$ such that for any two points $x,y$ in $\Omega_{i}(d) \equiv \Omega_i \cap \{y \in B_{3-d} \ | \ \delta_0(y) \geq d\}$, 
    $x,y$ are connected by a path $\gamma$, which is fully contained in $\Omega_i(\mu) \equiv \Omega_i \cap \{y \in B_{3-\mu} \ | \ \delta_0(y) \geq 
\mu \}$. The existence of such a $\mu$ and the Harnack inequality lead to the desired conclusion. In fact, if we use dyadic cubes of side length 
$\mu/10$ to cover $B_{10}$, those cubes which intersect with $\Omega_i(\mu)$ are fully contained in $\Omega_i(\mu /2)$, and the number of cubes are 
bounded by $Q \equiv C(n)\mu^{-n}$ (here $n$ is the dimension and not to be confused wth the subindices).

    Hence, when the subindices $n$ of $u_n$ is large enough, each $\Omega_i(\mu/2)$ will be fully contained in a single nodal domain of $u_n$ by Theorem 
\ref{Growth Estimate}. Since $\delta_0(y_n) \geq c/2$, $y_n$ is contained in a $\Omega_i(d)$ for an $i \in I$. $y_n$ and $x_i$ are then connected by a 
path $\gamma_{n,i}$ fully contained in $\Omega_i(\mu)$, which is covered by $Q$ cubes with side length $\mu/10$. We can then apply the Harnack inequality $Q$ times along $\gamma_{n,i}$, and get
    \begin{equation}\label{Pull 2}
        |v_n|(y_n) \leq h^{Q} |v_n|(x_i) \leq h^{Q} \sup_{B_{\eta - 4 \epsilon}}|v_n|
    \end{equation}
for some $h = h(\lambda)>1$.

    Combine (\ref{Expel Points Outward}) and (\ref{Pull 2}), we get that
        \begin{equation}
            \sup_{\eta-2\epsilon}|v_n| \leq M \cdot h^Q \sup_{\eta - 4\epsilon} |v_n|,
        \end{equation}
    which contradicts to (\ref{Blow Up Frequency}) by Theorem \ref{Monotone Frequency}.

\end{proof}

\section{Analysis on a Single Nodal Domain}\label{Section Single}

In this section, we will fix a single domain and discuss properties of solutions on this domain. 
More precisely,  let $L_0(u_0) = 0$ in $B_{10}$, $0 \in Z(u_0) = Z_0$ and $N_{u_0} \leq N_0 < \infty$, we consider a nodal domain 
$\Omega$ of $u_0$ in $B_5$ with $0 \in \pa \Omega$.
And we use the notation $\delta(x) \equiv \dist(x,Z_0) = \dist(x,\pa \Omega)$ for $x \in \Omega$.

\subsection{Boundary Harnack Inequality on a Given Nodal Domain}

\begin{theorem}\label{Boundary Harnack in Single Domain}
    \textit{
    Suppose that $L(u) = L(v) = 0$ in $\Omega$, $u> 0$ on $\Omega$, and $u=v=0$ on $\pa \Omega \cap B_3$ continuously. Then, there are positive constants $M = M(\lambda,N_0)$ and $r = r(\lambda,N_0)$, such that 
        \begin{equation}\label{Growth Control Single 1}
            \bigg| \frac{v}{u} \bigg|\leq M \cdot \frac{\sup_{B_1 \cap \Omega} |v|}{\inf_{ y \in B_2 \cap \Omega, \ \delta(y) \geq r} u(y)},
        \end{equation}
    on $B_{1/4} \cap \Omega$.
    In particular, if $v>0$ on $\Omega$, and
        \begin{equation}
            0 < C_1 \leq v, u \leq C_2
        \end{equation}
    on $\{x \in B_2 \cap \Omega \ | \ \delta(x) \geq r\}$, then
        \begin{equation}\label{Growth Control Single 3}
            \frac{C_1}{C_2} \cdot M^{-2} \leq \frac{v}{u} \leq \frac{C_2}{C_1} \cdot M^2
        \end{equation}
    on $ B_{1/4} \cap \Omega.$
    }
\end{theorem}
We prove this theorem with cubes in the place of balls for conveniences. 
We consider cubes $Q_s$ with center $0$ 
and side-length $2s$, and 
we denote $K_s \equiv \Omega \cap Q_s$, $A_s \equiv \{ x\in K_s \ | \ 
\delta(x) \geq \delta \cdot s\}$, where $\delta = \delta(\lambda,N_0) \ll 
1$ will be chosen in the following lemma. The argument is inspired by 
\cite{DS20} for NTA domains. 

\begin{lemma}\label{Iteration Decay}
    \textit{
    There exist $M_0 = M_0(\lambda,N_0)>0$ and $\delta = \delta(\lambda,N_0)>0$, such that if $w$ is a solution to $L(w) =0$ in $K_1$, possibly change sign, which vanishes on $\pa \Omega \cap B_1$, and $w \geq M_0$ on $A_1$, $w \geq -1$ on $K_1$, then we will have that $w \geq M_0 \cdot a$ on $A_{1/2}$, $w \geq -a$ on $K_{1/2}$ for some small $a = a(\lambda,N_0)>0$.
    }
\end{lemma}
\begin{proof}
    First, we construct the lower bound on $A_{1/2}$. Pick an $x_0 \in A_{1/2}$, then, there is a modified Harnack chain $\{x_0, x_1, \dots, x_m\}$, which we got in Theorem \ref{Stubborn Ball}. In the Corkscrew Condition of Theorem \ref{Stubborn Ball}, we showed that $\delta(x_m) \geq c_4 = c_4 (\lambda,N_0)$. We also got that,
        \begin{equation}
            m \leq -\xi_1 \log(\delta(x_0)) + \xi_2 \leq -\xi_1 \log(\delta/2) + \xi_2 ,
        \end{equation}
    in (iv) of Theorem \ref{Stubborn Ball} for $\xi_1 = \xi_1 (\lambda, N_0) >0$, $\xi_2 =  
\xi_2(\lambda, N_0)$. Hence, if we assume that $\delta < c_4$ first, by the Harnack inequality 
along this chain with constant $h(\lambda,\theta_0) = h(\lambda,N_0)>1$, we have
        \begin{equation}
            w(x_0) \geq (M_0+1) \cdot h^{\xi_1 \log(\delta/2) - \xi_2} - 1 . 
        \end{equation}
    We choose $a = (1/2) \cdot h^{\xi_1 \log(\delta/2) - \xi_2} $. Then, when $M_0 \geq 1/a$, we have that
        \begin{equation}
            w(x_0) \geq M_0 \cdot a. 
        \end{equation}
    
    Then, we will show that $w \geq -a$ on $K_{1/2}$ for suitable $\delta$. Let $x_0 \in K_{1-2\delta} \backslash A_1$, consider the cube $Q(x_0,2 \delta)$. By Theorem \ref{Stubborn Ball}, there is a small ball with radius $\delta \cdot c$ for some $c = c(\lambda, N_0)$ in $Q(x_0,2 \delta) \backslash \Omega$, where $w^{-} = 0$. 
    Hence, by the weak Harnack inequality we mentioned in the proof of Lemma \ref{Caffarelli's Lemma}, and $w \geq -1$ on $K_1$,
        \begin{equation}
            w^{-}(x_0) \leq (1-c_1) \sup_{Q(x_0,2 \delta)} w^{-} \leq (1-c_1),
        \end{equation}
    for some $c_1 = c_1(\lambda,N_0) \in (0,1)$. Hence, $w^{-} \leq (1-c_1)$ in $K_{1-2\delta}$. By iteration, we get $w^{-} \leq (1-c_1)^t$ in $K_{1-2t \delta}$, and then,
        \begin{equation}
            w \geq -(1-c_1)^{\frac{1}{8\delta}}, \ \mr{on} \  K_{1/2}.
        \end{equation}
    Since $\delta \cdot \log(\delta) \to 0$ as $\delta \to 0$, we can choose a small $\delta = \delta(\lambda, N_0)$, such that
        \begin{equation}
            (1-c_1)^{\frac{1}{8\delta}} \leq a = (1/2) \cdot h^{\xi_1 \log(\delta/2) - \xi_2} .
        \end{equation}
\end{proof}

By the above Lemma \ref{Iteration Decay}, and by iterating the same arguments on $K_{2^{-t}}$ and 
$A_{2^{-t}}$, we can conclude that $w > 0$ in $\{x \in K_1 \ | \ \delta(x) \geq 2 \delta |x|\}$. 
Because one can vary the centers of $K_s$ and $A_s$, it leads to  $w > 0$ in $K_{1/4}$.

\begin{proof}[Proof of Theorem \ref{Boundary Harnack in Single Domain}]
    Set $w \equiv Cu-v$. We will choose suitable $C$ so that $w$ will satisfy the assumptions of Lemma \ref{Iteration Decay}.

    For (\ref{Growth Control Single 1}), the statement follows by choosing $M \geq M_0 + 1$, $r \leq \delta$ with $M_0, \delta$ in Lemma \ref{Iteration Decay}, and choosing $C = M \cdot \sup_{K_1} |v| \cdot (\inf_{ y \in K_2, \ \delta(y) \geq r} u(y))^{-1}$

    For (\ref{Growth Control Single 3}), by Lemma \ref{Caffarelli's Lemma},
        \begin{equation}
            \sup_{K_1}v \leq M_1 \cdot \sup_{y\in K_{3/2}, \ \dist(y,Z_0) \geq c} v,
        \end{equation}
    for some positive $M_1 =M_1(\lambda,N_0)$ and $c = c(\lambda,N_0)$. Hence, if we choose $r = \min\{c, \delta\}$, $M = \max\{M_1 , M_0+1\}$, with $M_0, \delta$ in Lemma \ref{Iteration Decay}, we have
        \begin{equation}
            \sup_{K_1}v \leq M \cdot C_2.
        \end{equation}
    Then, we choose $C=C_1 ^{-1} C_2 M^2 $ and the conclusions of the theorem follow.
\end{proof}

\begin{remark}
    The strong maximum principle holds for $v/u$ by a similar proof as in Theorem \ref{Strong 
 Maximum Principle}. An interesting part is that $\sup_{\Omega \cap B_1} v/u$ may not be achieved 
 on $\pa \Omega \cap B_1$.
 \end{remark}

\begin{corollary}\label{Growth Control on a Single Domain}
    \textit{
    Suppose that $L_0(v) = 0$ in $\Omega$, $v>0$ on $\Omega$, $v=0$ on $\pa \Omega \cap B_3$ continuously. 
    Then, there are positive constants $C = C(\lambda, N_0)$, $r = r(\lambda,N_0)$, such that
        \begin{equation}
            C^{-1}  \cdot \delta(x)^{N_0} \cdot \inf_{y \in B_2 \cap \Omega, \ \delta(y) 
\geq r} v(y)\leq v(x)\leq  C \cdot \delta^{\alpha}(x) \cdot \sup_{B_1 \cap \Omega} v
        \end{equation}
    in $B_{1/4} \cap \Omega$.
    }
\end{corollary}
\begin{proof}
    These follow from considering $v$ and $u_0$ in Theorem \ref{Boundary Harnack in Single Domain} and Theorem \ref{Growth Estimate}.
\end{proof}

\begin{remark}\label{Neck Region}
    As we can see in (\ref{Growth Control Single 3}) of Theorem \ref{Boundary Harnack in Single Domain}, the upper bound will depend on the ratio $C_2/C_1$, which also depends on the set $\{ x \in K_2 \ | \ \delta(x)\geq r\}$. Since $\Omega$ is connected, one can apply the usual Harnack inequality on this set. So, $C_2/C_1$ is actually a quantity depending on the shape
    of the single nodal domain $\Omega$. Is it possible that $C_2/C_1$ could be controlled by some constants only depending on $N_0$? The answer is no and we have the following counter example.
    Consider $\Omega_{\epsilon} \equiv \{(x,y) \in \mb{R}^2 \ | \ x^2 -y^2 > -\epsilon, |x| < 
1\}$, which is the part of one nodal domain of $u_\epsilon (x,y) = x^2 -y^2 + \epsilon$ in 
$B_1$. It has a thin and short neck region around the origin. Let $v_\epsilon$ be the 
solution of the following Dirichlet problem:
        \begin{equation}
            \Delta v_\epsilon = 0 , \ \mr{on} \ \Omega_\epsilon,
        \end{equation}
and
        \begin{equation}
            v_{\epsilon} = 1, \  \mr{on} \ \{x = 1\} \cap \pa \Omega_\epsilon \ ; \ v_{\epsilon} = 0, \ \mr{on}  \ \{-1\leq x < 1\} \cap \pa \Omega_\epsilon .
        \end{equation}
We notice that $v_\epsilon > 0$ in $\Omega_\epsilon$, $v_\epsilon > C(n)>0$ when $x >1/2$ and 
$|y| < 1/2$, and $v_\epsilon$ is very close to $0$ when $x < 0$. As $\epsilon \to 0$, 
$v_\epsilon$ will tend to $0$ on $\{x <0\}$, which means that $u_\epsilon / v_\epsilon \to 
\infty$. One can also consider $v_\epsilon(-x, y)$ for similar purpose on the part of nodal 
domain with $x>0$.  If one replaces $y$ by $(y_1,..., y_{n-1})$ one find examples in dimension 
$n$. On the other hand, if one replaces $x$ by $(x_1,x_2)$ and $y$ by $(y_1,y_2)$, then there is
no problem when $\epsilon$ goes to zero. In the latter case, $\Omega_\epsilon$ is quantitatively 
connected (independent of small $\epsilon$), see Definition \ref{Quantitative Connectedness}. 
\end{remark}

Nodal domains are path-connected by its definition. Examples in the Remark \ref{Neck Region} showed that they can easily be degenerate and decompose into several smaller nodal domains even for a sequence of nodal domains of solutions in $\mc{S}_{N_0}(\Lambda)$.
Consequently, many analytic estimates on a nodal domain of a solution $u_0 \in \mc{S}_{N_0} (\Lambda)$ are not uniform (depending only on $\Lambda$ and $N_0$).
On the other hand, even a single nodal domain is degenerate and decompose into several smaller nodal domains, the number of such small nodal domains is again locally uniformly bounded by a constant $T_0(\lambda,N_0)$, see Corollary \ref{Number of Domains}.

Inspired by our proof of Theorem \ref{Bound Frequency} and Lemma \ref{Caffarelli's Lemma}, if we use dyadic cubes with side length $r/10$ to cover $B_{10}$ with $r = r(\lambda,N_0)$ chosen in the Theorem \ref{Boundary Harnack in Single Domain}, those cubes which have nonempty intersections with $\{ y\in B_2 \ | \ \delta(y) \geq r\}$ will form several big chunks $E_i$, $i = 1, \dots, T $, with each $E_i$ path-connected and $T \leq C(n)r^{-n}$, but different $E_i, E_j$ are disjoint. 
Then, we could give another interesting upper bound in the following.

\begin{corollary}\label{Big Chunk Bound}
    \textit{
    Suppose that $L_0(v) = 0$ in $\Omega$, $v>0$ on $\Omega$, $v=0$ on $\pa \Omega \cap B_3$ continuously. Then, there is a positive constant $C = C(\lambda,N_0)$, such that
        \begin{equation}\label{Growth Control Single 2}
            \frac{v(x)}{u_0(x)} \leq C \cdot \max\bigg\{\frac{v(x_1)}{u_0(x_1)} , \dots, \frac{v(x_T)}{u_0(x_T)} \bigg\},
        \end{equation}
    in $B_{1/4} \cap \Omega$, where $x_i$ is an arbitrary point inside $E_i$ for each $i = 1, \dots, T$.
    }
\end{corollary}
\begin{proof}
    On the right hand side of (\ref{Growth Control Single 1}), by Lemma \ref{Caffarelli's Lemma}, for some $M = M(\lambda,N_0)$ we have that
        \begin{equation}
            \sup_{B_1 \cap \Omega }v \leq M \cdot \sup_{y \in B_{3/2} \cap \Omega , \ \delta(y)\geq r} v.
        \end{equation}
    Assume the maximal value on the right hand side of the above inequality is achieved by a point $y_1 \in E_1$. Then, by the usual Harnack inequality inside $E_1$ and the fact that the number of all dyadic cubes is also bounded by $C(n)r^{-n}$, there is a constant $C = C(\lambda, N_0)$ such that
        \begin{equation}
            v(y_1) \leq C \cdot v(x_1).
        \end{equation}
    By Theorem \ref{Growth Estimate} and Harnack inequality, there is a constant $c = c(\lambda,N_0,r) = c(\lambda,N_0)$ such that
        \begin{equation}
            \inf_{ y \in B_2 \cap \Omega, \ \delta(y)) \geq r} u_0(y) \geq c \cdot u_0 (x_1).
        \end{equation}
    By combining the above three inequalities and (\ref{Growth Control 
Single 1}), we obtain (\ref{Growth Control Single 2}).
\end{proof}

The above discussions inspire one to introduce the notion of the Quantitative Connectedness in the following.
\begin{definition}\label{Quantitative Connectedness}
    We say that the nodal domain $\Omega$ is quantitatively connected, if there are positive constants $\delta_1 = \delta_1(\lambda,N_0) \leq \delta_2 = \delta_2(\lambda,N_0) \leq r/2 $,
    such that for any $x_0 \in \overline{\Omega}$, any $s>0$, and any pair of points $x,y$ in $\Omega \cap B_s (x_0)$ with $\delta (x) \geq \delta_2 \cdot s$, $\delta (y) \geq \delta_2 \cdot s$, could be connected by a path totally
    contained inside $\Omega \cap B_s(x_0) \cap \{z \ | \ \delta(z) \geq s \cdot \delta_1\}$.
\end{definition}

If $\Omega$ is quantitatively connected, it is then easy to show that in Corollary \ref{Big Chunk Bound}, one could give an upper bound by an arbitrary $v(x_i) / u_0 (x_i)$. And with the assumptions in Theorem \ref{Boundary Harnack in Single Domain}, one could show the H\"{older} continuity of $v/u$ to the boundary $\pa \Omega$ if $\Omega$ is quantitatively connected.

\subsection{Some Other Properties and Connections to Other Typical Domains}

Apart from the corkscrew property and the modified Harnack chain obtained in Section \ref{Section Corkscrew}, the nodal domain $\Omega$ of a solution $u_0 \in \mc{S}_{N_0}(\Lambda)$ has several other properties that are important for classical potential analysis on non-smooth domains. Let us recall a frew of such properties here.

\begin{property}
    For $u_0 \in \mc{S}_{N_0} (\Lambda)$, $\pa \Omega \cap B_{5}$ is Ahlfors regular.
    Indeed, the upper bound   
        \begin{equation}
            H^{n-1}(B_s (x) \cap \pa \Omega ) \leq H^{n-1}(B_s (x) \cap Z(u_0) ) \leq C(\Lambda,N_0) \cdot s^{n-1}
        \end{equation}
    for all $x \in \pa \Omega \cap B_5$ and $r \in (0,1)$, follows from the geometric measure estimates (\ref{Measure Zero Set}), see for example \cite{DF88}, \cite{HS89}, \cite{L91}, \cite{HL00}, \cite{CNV15}, \cite{NV17}, \cite{LM18} and \cite{L18}.
    The lower bound follows from the corkscrew condition that 
        \begin{equation}
            | \Omega \cap B_s(x) | \geq C(\Lambda,N_0) \cdot s ^n \ ,
        \end{equation}
        \begin{equation}
            | \Omega ^c \cap B_s(x) | \geq C(\Lambda,N_0) \cdot s ^n \ ,
        \end{equation}
    for some $C(\Lambda,N_0) >0$ and the relative isoperimetric inequality
        \begin{equation}
            H^{n-1}(\pa \Omega \cap B_s (x)) \geq C(n) \cdot \big(\min \{|\Omega \cap B_s(x)| , \  |\Omega ^c \cap B_s(x)| \} \big) ^{\frac{n-1}{n}} .
        \end{equation}
\end{property}

It is also clear from the proofs in \cite{HS89}, \cite{HL00} that the following is true.

\begin{property}
    $\pa \Omega  \cap B_5$ is uniformly rectifiable.
    In fact, there is an $\epsilon_0 = \epsilon_0 (\Lambda,N_0) >0$, such that for any $\epsilon \in (0,\epsilon_0)$, $Z(u_0)\cap B_5$ can be decomposed into two parts. One big part is a $C^1$-hypersurface with $C^1$-structure depending on $\epsilon$, and the other small part has $H^{n-1}$ Hausdorff measures less than $\epsilon$.
\end{property}

Finally, we examine some basic properties of harmonic measures with poles in $\Omega$. For any pole $x_0 \in \{ x \in \Omega \cap B_2\ | \ \delta(x) \geq r/2\}$ with $r = r(\lambda ,N_0)$ chosen in Theorem \ref{Boundary Harnack in Single Domain}, one can easily show that 
    \begin{equation}
        G(x_0 , x) \leq C(\lambda , N_0) \cdot  u_0(x)
    \end{equation}
for $x \in \Omega \cap B_1$ by the maximum principle on $(\Omega \cap B_5) \backslash \overline{B_{r/4}(x_0)}$. Here, $G(x_0 , \cdot)$ is the Green function of $L_0$ on $\Omega \cap B_5$.
Hence, by the definition of harmonic measure, we have that
    \begin{equation}
        \omega_{x_0}  \llcorner (\pa \Omega \cap B_1) \leq C(\lambda , N_0) \cdot | \nabla u_0 | \cdot H^{n-1} \llcorner (\pa \Omega \cap B_1).
    \end{equation}
Here $\omega_{x_0}$ is the Harmonic measure on $\pa (\Omega \cap B_5)$ with pole $x_0$.

By estimate (\ref{Measure Singular Set}) and the gradient estimates for $u_0$, one sees that
    \begin{equation}
        | \nabla u_0 | \cdot H^{n-1} \llcorner (\pa \Omega \cap B_2) \ll   H^{n-1} \llcorner (\pa \Omega \cap B_2) \ll | \nabla u_0 | \cdot H^{n-1} \llcorner (\pa \Omega \cap B_2).
    \end{equation}
On the other hand, by Corollary \ref{Big Chunk Bound} and Lemma \ref{Caffarelli's Lemma}, one can prove that
    \begin{equation}
        \sum_{i=1} ^T \omega_{x_i}  \llcorner (\pa \Omega \cap B_2) \geq C(\lambda ,N_0) \cdot |\nabla u_0 | \cdot H^{n-1} \llcorner (\pa \Omega \cap B_2).
    \end{equation}
Hence, we conclude the following.

\begin{theorem}
    \textit{
    Let $\Omega$ be a nodal domain of a solution $u_0 \in \mc{S}_{N_0}(\Lambda)$ with $0 \in \pa \Omega$. Then, there is a set of points $\{x_1 , \dots , x_{T}\}$ chosen in Corollary \ref{Big Chunk Bound} with $T \leq T_0 (\lambda,N_0) < \infty$ in $\Omega \cap B_2$, such that
        \begin{equation}
            C^{-1} \cdot |\nabla u_0 | \cdot H^{n-1} \llcorner (\pa \Omega \cap B_1) \leq \sum_i ^{T} \ \omega_i \llcorner (\pa \Omega \cap B_1)  \leq C \cdot | \nabla u_0 |  \cdot H^{n-1} \llcorner (\pa \Omega \cap B_1) 
        \end{equation}
    for some positive constant $C = C(\lambda, N_0)$, where $\{ \omega_i (\cdot)\}$ are harmonic measures on 
$\pa (\Omega \cap B_5)$ with poles $x_i \in \Omega \cap B_2$, for $i= 1,2, \dots, T$. In particular, 
$\sum_i ^{T} \ \omega_i 
\llcorner (\pa \Omega \cap B_1)$, $H^{n-1} \llcorner (\pa \Omega \cap B_1) $ and $|\nabla u_0| \cdot H^{n-1} \llcorner (\pa \Omega \cap B_1) $
    are mutually absolutely continuous.
    }
\end{theorem}

\end{document}